\documentclass{birkjour}
\usepackage{amsthm}
\usepackage{amsfonts}
\usepackage{mathrsfs}
\usepackage{amscd,xy,amssymb,amsmath,graphicx,verbatim}
\usepackage{amssymb, color}
\usepackage[bookmarksnumbered, colorlinks, plainpages]{hyperref}

\newtheorem{theorem}{Theorem}[section]
\newtheorem{lemma}[theorem]{Lemma}
\newtheorem{corollary}[theorem]{Corollary}

\newtheorem{definition}[theorem]{Definition}
\newtheorem{example}[theorem]{Example}

\theoremstyle{remark}

\numberwithin{equation}{section}

\begin{document}

%
%
%
%
%
%
%
%
%


\title[Cowen-Douglas function and its application on chaos]{Cowen-Douglas function and its application on chaos}
\author[Lvlin Luo]{Lvlin Luo}
\address{School of Mathematics and Statistics, Xidian University, 710071, Xi'an, P. R. China}
\address{School of Mathematical Sciences, Fudan University, 200433, Shanghai, P. R. China}
\email{luoll12@mails.jlu.edu.cn}
\email{luolvlin@fudan.edu.cn}






\begin{abstract}
In this paper,
on $\mathbb{D}$ we define Cowen-Douglas function introduced by Cowen-Douglas operator $M_\phi^*$ on Hardy space $\mathcal{H}^2(\mathbb{D})$ and we give a sufficient condition for Cowen-Douglas function,
where $\phi\in\mathcal{H}^\infty(\mathbb{D})$.
Moreover, we give some applications of Cowen-Douglas function on chaos,
such as application on the inverse chaos problem for $\phi(T)$,
where $\phi$ is a Cowen-Douglas function and $T$ is the backward shift operator on $\mathcal{L}^2(\mathbb{N})$.
\end{abstract}
\maketitle

\section{Introduction}

If $X$ is a metric space and $T$ is a continuous self-map on $X$,
then the pair $(X,T)$ is called a topological dynamic systems, which is induced by the iteration
$$T^n=\underbrace{T\circ\cdots\circ T}\limits_{n},\qquad n\in\mathbb{N}.$$

If the metric space $X$ and the continuous self-map $T$ are both linear,
then the topological dynamic systems $(X,T)$ is named linear dynamic.
Especially,
if $T$ is a continuous invertible self-map on $X$,
then $(X,T)$ is named invertible dynamic.
Similarly, for a Banach space $\mathbb{B}$,
if $T$ is a invertible bounded linear operator on $\mathbb{B}$,
then the invertible dynamic $(X,T)$ is named invertible linear dynamic.
Also, to study the chaos theory of invertible linear dynamic we named that is inverse chaos problem.

For invertible dynamic,
the relationship between $(X,f)$ and $(X,f^{-1})$ is raised by Stockman as an open question\cite{DStockman2012}.
And \cite{LuoLvlinHouBingzhe2015,HouBingzheLuoLvlin2016} and \cite{LuoLvlinHouBingzhe2016} give counter examples for Li-Yorke chaos on noncompact space and compact space, respectively.
However, there is no generally way for this research.

Hardy space $\mathcal{H}^2(\mathbb{D})$ is a special Hilbert space that is relatively nice,
not only for theoretical but also for computing through complex analytic functions.
Because of any separable infinite-dimension Hilbert space are isomorphic,
to construct some symbol operators on Hardy space to study the inverse chaos problem may be a good idea.

In this paper,
on $\mathbb{D}$ we define Cowen-Douglas function introduced by Cowen-Douglas operator $M_\phi^*$ on Hardy space $\mathcal{H}^2(\mathbb{D})$ and we give a sufficient condition for Cowen-Douglas function,
where $\phi\in\mathcal{H}^\infty(\mathbb{D})$.

In the last, we give some applications of Cowen-Douglas function on chaos,
such as application on the inverse chaos problem for $\phi(T)$,
where $\phi$ is a Cowen-Douglas function and $T$ is the backward shift operator on $\mathcal{L}^2(\mathbb{N})$,
i.e.,
$$T(x_1,x_2,\cdots)=(x_2,x_3,\cdots).$$

\section{Some properties of Hardy space and chaos}

Firstly, we give some properties of Hardy space.
For $\mathbb{D}=\{z\in\mathbb{C},|z|<1\}$,
if $g$ is a complex analytic function on $\mathbb{D}$ and there is
$$
\sup\limits_{r<1}\int_{-\pi}^{\pi}|g(re^{i\theta})|^2\,d\theta<+\infty,
$$
then we denote $g\in\mathcal{H}^2(\mathbb{D})$.
Obviously,
$\mathcal{H}^2(\mathbb{D})$ is a Hilbert space with the norm
$$
\|g\|_{\mathcal{H}^2}^2=\sup\limits_{r<1}\int_{-\pi}^{\pi}|g(re^{i\theta})|^2\,\frac{d\theta}{2\pi},
$$
and $\mathcal{H}^2(\mathbb{D})$ is denoted as Hardy space.

For any given complex analytic function $g$,
we get the Taylor expansion
$$
g(z)=\sum\limits_{n=0}^{+\infty}a_nz^n,
$$
so the relation between $g\in\mathcal{H}^2(\mathbb{D})$ and $\sum\limits_{n=0}^{+\infty}a_n^2<+\infty$ are naturally equivalent.

Let $\mathbb{T}=\partial\mathbb{D}$,
if $\mathcal{H}^2(\mathbb{T})$ denoted the closed span of Taylor expansions of functions in $\mathcal{L}^2(\mathbb{T})$,
then $\mathcal{H}^{2}(\mathbb{T})$ is a closed subspace of $\mathcal{L}^2(\mathbb{T})$.
From the naturally isomorphic between $\mathcal{H}^{2}(\mathbb{D})$ and $\mathcal{H}^{2}(\mathbb{T})$ by the properties of analytic function,
we denote $\mathcal{H}^{2}(\mathbb{T})$ also as Hardy space.

Let $\mathcal{H}^{\infty}(\mathbb{D})$ denote the set of all bounded complex analytic functions on $\mathbb{D}$,
and $\mathcal{H}^{\infty}(\mathbb{T})$ denote the closed span of Taylor expansions of functions in $\mathcal{L}^{\infty}(\mathbb{T})$.
Then we get that $\mathcal{H}^{\infty}(\mathbb{D})$ and $\mathcal{H}^{\infty}(\mathbb{T})$ are naturally isomorphic by
the properties of complex analytic functions \cite{HenriCartanyujiarong2008}P55P97 associated with the Dirichlet Problem
\cite{HenriCartanyujiarong2008}P103.

For any given $\phi\in\mathcal{H}^{\infty}(\mathbb{D})$,
it is easy to get that $\|\phi\|_{\infty}=\sup\{|\phi(z)|;|z|<1\}$ is a norm on $\mathcal{H}^{\infty}(\mathbb{D})$.
And for any given $g\in\mathcal{H}^{2}(\mathbb{D})$,
the multiplication operator $M_{\phi}(g)=\phi g$ associated with $\phi\in\mathcal{H}^{\infty}(\mathbb{D})$ is a bounded linear operator,
and on $\mathcal{H}^{2}(\mathbb{D})$ we get

$$
\|M_{\phi}(g)\|\leq\|\phi\|_{\infty}\|g\|_{\mathcal{H}^2}.
$$

By \cite{JohnBConway1990}P6 we get that any Cauchy sequence with the norm $\|\cdot\|_{\mathcal{H}^2}$ on $\mathcal{H}^2(\mathbb{D})$ is a uniformly Cauchy sequence on any closed disk in $\mathbb{D}$,
in particular, we get that the point evaluations $f\to f(z)$ are continuous linear functional on $\mathcal{H}^2(\mathbb{D})$,
by the Riesz Representation \cite{JohnBConway1990}P13,
for any $g(s)\in\mathcal{H}^2(\mathbb{D})$,
there is a unique $f_z(s)\in\mathcal{H}^2(\mathbb{D})$ such that

$$g(z)=<g(s),f_z(s)>,$$

and then define this $f_z$ is a reproducing kernel at $z\in\mathbb{D}$.
There are more properties about Hardy space in
\cite{FBayartEMatheron2009}P7,
\cite{JohnBGarnett2007}P48,
\cite{KennethHoffman1962}P39,
\cite{RonaldGDouglas1998}P133 and
\cite{WilliamArveson2002}P106.

For prepare, we give some definition and properties of chaos on Hilbert space
$\mathbb{H}$,
let $\mathcal{B}(\mathbb{H})$ denote the set of all bounded linear operator on $\mathbb{H}$.

\begin{definition}\label{liyorkehundundedingyi1}
Let $T\in\mathcal{B}(\mathbb{H})$,
if there exists $x\in\mathbb{H}$ satisfies:

\begin{eqnarray*}
&&(1)\varlimsup\limits_{n\to\infty}\|T^{n}(x)\|>0, \\
&&(2)\varliminf\limits_{n\to\infty}\|T^{n}(x)\|=0,
\end{eqnarray*}

then we say that $T$ is Li-Yorke chaotic,
and named $x$ is a Li-Yorke chaotic point of $T$, where $x\in\mathbb{H},n\in\mathbb{N}$.
\end{definition}

\begin{definition}\label{cowendouglassuanzidingyi2}
For a connected open subset $\Omega\subset\mathbb{C}$,
$n\in\mathbb{N}$,
let $\mathcal{B}_n(\Omega)$ denote the set of all bounded linear operator $T$ on $\mathbb{H}$ that satisfies:

(1) $\Omega\in\sigma(T)=\{\omega\in\mathbb{C}:T-\omega \text{ is not invertible}\}$;

(2) $ran(T-\omega)=\mathbb{H}$ for $\omega\in\Omega$;

(3) $\bigvee\ker\limits_{\omega\in\Omega}(T-\omega)=\mathbb{H}$;

(4) $\dim\ker(T-\omega)=n$ for $\omega\in\Omega$.

If $T\in\mathcal{B}_n(\Omega)$, then we say that $T$ is a Cowen-Douglas operator.
\end{definition}

\begin{theorem}[\cite{BHouPCuiYCao2010}]\label{cowendouglassuanzidingyidevaneyliyrokefenbuqianghunhedingli2}
For a connected open subset $\Omega\subset\mathbb{C}$,
$T\in\mathcal{B}_n(\Omega)$,
we get

(1) If $\Omega\bigcap\mathbb{T}\neq\emptyset$,
then $T$ is Devaney chaotic.

(2) If $\Omega\bigcap\mathbb{T}\neq\emptyset$,
then $T$ is distributionally chaotic.

(3) If $\Omega\bigcap\mathbb{T}\neq\emptyset$,
then $T$ strong mixing.
\end{theorem}

\section{Cowen-Douglas function on Hardy space}

In this part,
we give the definition of rooter function for a given $m$-folder complex analytic function on $\mathbb{D}$.
Then we define Cowen-Douglas function,
and we prove that if the rooter function of a $m$-folder complex analytic function is an outer function,
then the $m$-folder complex analytic function is a Cowen-Douglas function.

\begin{definition}[\cite{RonaldGDouglas1998}P141]\label{waihanshudingyi46}
Let $\mathcal{P}(z)$ be the set of all polynomials about $z$,
where $z\in\mathbb{T}$.
Define a function $h(z)\in\mathcal{H}^2(\mathbb{T})$ is an outer function if
$$cl[h(z)\mathcal{P}(z)]=\mathcal{H}^2(\mathbb{T}).$$
\end{definition}

\begin{lemma}[\cite{RonaldGDouglas1998}P141]\label{waihanshukeyidexingzhi49}
A function $h(z)\in\mathcal{H}^{\infty}(\mathbb{T})$ is invertible on the Banach algebra $\mathcal{H}^{\infty}(\mathbb{T})$,
if and only if
$h(z)\in\mathcal{L}^{\infty}(\mathbb{T})$
and $h(z)$ is an outer function.
\end{lemma}

\begin{theorem}[\cite{JohnBGarnett2007}P81]\label{chengfasuanzishimanshedangqiejindangshiwaihanshu45}
Let $\mathcal{P}(z)$ be the set of all polynomials about $z$,where $z\in\mathbb{D}$.
then $h(z)\in\mathcal{H}^2(\mathbb{D})$ is an outer function if and only if $$\mathcal{P}(z)h(z)=\{p(z)h(z);p\in\mathcal{P}(z)\}\text{ is dense in }\mathcal{H}^2(\mathbb{D}).$$
\end{theorem}

Let $\phi$ be a non-constant complex analytic function on $\mathbb{D}$,
for any given $z_0\in\mathbb{D}$,
by \cite{HenriCartanyujiarong2008}P29 we get that there exists $\delta_{z_0}>0$,
exists $k_{z_0}\in\mathbb{N}$, when $|z-z_0|<\delta_{z_0}$, there is
\begin{eqnarray*}
\phi(z)-\phi(z_0)=(z-z_0)^{k_{z_0}}h_{z_0}(z)
\end{eqnarray*}

where $h_{z_0}(z)$ is complex analytic on a neighbourhood of $z_0$ and $h_{z_0}(z_0)\neq0$.

\begin{definition}\label{nyejiexihanshudedingyijitoubujiaobu51}
Let $\phi$ be a non-constant complex analytic function on $\mathbb{D}$,
for any given $z_0\in\mathbb{D}$, there exists $\delta_{z_0}>0$ such that
\begin{eqnarray*}
\phi(z)-\phi(z_0)|_{|z-z_0|<\delta_{z_0}}=p_{n_{z_0}}(z)h_{z_0}(z)|_{|z-z_0|<\delta_{z_0}},
\end{eqnarray*}

$h_{z_0}(z)$ is complex analytic on a neighbourhood of $z_0$ and $h_{z_0}(z_0)\neq0$,
$p_{n_{z_0}}(z)$ is a $n_{z_0}$-th polynomial and the $n_{z_0}$-th coefficient is equivalent $1$.
By the Analytic Continuation Theorem \cite{HenriCartanyujiarong2008}P28,
we get that there is a unique complex analytic function $h_{z_0}(z)$ on $\mathbb{D}$ such that
$\phi(z)-\phi(z_0)=p_{n_{z_0}}(z)h_{z_0}(z)$,
then define that $h_{z_0}(z)$ is a rooter function of $\phi$ at $z_0$.
If for any given $z_0\in\mathbb{D}$,
the rooter function $h_{z_0}(z)$ has non-zero point but the roots of $p_{n_{z_0}}(z)$ are all in $\mathbb{D}$ and $n_{z_0}\in\mathbb{N}$ is a constant $m$ on $\mathbb{D}$,
then define that $\phi$ is a $m$-folder complex analytic function on $\mathbb{D}$.
\end{definition}

\begin{definition}\label{youjiejiexihanshuliCowenDouglashanshudingyi47}
Let $\phi(z)\in\mathcal{H}^{\infty}(\mathbb{D}),n\in\mathbb{N}$,
$M_{\phi}$ is the multiplication by $\phi$ on $\mathcal{H}^{2}(\mathbb{D})$.
If the adjoint multiplier $M_{\phi}^{*}\in\mathcal{B}_n(\bar{\phi}(\mathbb{D}))$,
then define $\phi$ is a Cowen-Douglas function.
\end{definition}

By Definition $\ref{youjiejiexihanshuliCowenDouglashanshudingyi47}$ we get that any constant complex analytic function is not a Cowen-Douglas function

\begin{theorem}\label{cowendouglussuanziyujiexichengfasuanzijigezhonghundundeguanxi24}
Let $\phi(z)\in\mathcal{H}^{\infty}(\mathbb{D})$ be a $m$-folder complex analytic function,
$M_{\phi}$ is the multiplication by $\phi$ on $\mathcal{H}^{2}(\mathbb{D})$.
If for any given $z_0\in\mathbb{D}$,
the rooter functions of $\phi$ at $z_0$ is a outer function,
then $\phi$ is a Cowen-Douglas function,
that is,
the adjoint multiplier $M_{\phi}^{*}\in\mathcal{B}_m(\bar{\phi}(\mathbb{D}))$.
\end{theorem}
\begin{proof}
By the definition of $m$-folder complex analytic function Definition $\ref{nyejiexihanshudedingyijitoubujiaobu51}$ we get that $\phi$ is not a constant complex analytic function.
For any given $z\in\mathbb{D}$,
let $f_z\in\mathcal{H}^2(\mathbb{D})$ be the reproducing kernel at $z$.
We confirm that $M_{\phi}^{*}$ is valid the conditions of Definition $\ref{cowendouglassuanzidingyi2}$ one by one.

(1) For any given $z\in\mathbb{D}$,
$f_z$ is an eigenvector of $M_{\phi}^{*}$ associated with eigenvalue $\lambda=\bar{\phi}(z)$.
Because for any $g\in\mathcal{H}^2(\mathbb{D})$ we get
\begin{eqnarray*}
<g,M_{\phi}^{*}(f_z)>_{\mathcal{H}^2}=<\phi g,f_z>_{\mathcal{H}^2}=\phi(z)f(z)=<g,\bar{\phi}(z)f_z>_{\mathcal{H}^2}
\end{eqnarray*}

By the Riesz Representation Theorem\cite{JohnBConway1990}P13 of bounded linear functional in the form of inner product on Hilbert space, we get
$$M_{\phi}^{*}(f_z)=\bar{\phi}(z)f_z=\lambda f_z,$$
that is,
$f_z$ is an eigenvector of $M_{\phi}^{*}$ associated with eigenvalue $\lambda=\bar{\phi}(z)$.

(2) For any given $\bar{\lambda}\in\mathbb{\phi(D)}$,
because of $0\neq\phi\in\mathcal{H}^{\infty}(\mathbb{D})$,
we get that the multiplication operator $M_{\phi}-\lambda$ is injection by the properties of complex analysis,
hence
$$\ker({M_{\phi}-\lambda})=0.$$

Because of
$$\mathcal{H}^2(\mathbb{D})=\ker(M_{\phi}-\lambda)^{\bot}=cl[ran(M_{\phi}^{*}-\bar{\lambda})],$$
we get that
$ran(M_{\phi}^{*}-\bar{\lambda})$ is a second category space.
By \cite{JohnBConway1990}P305 we get that $M_{\phi}^{*}-\bar{\lambda}$ is a closed operator,
also by \cite{JohnBConway1990}P93 or \cite{zhanggongqinglinyuanqu2006}P97 we get
$$ran(M_{\phi}^{*}-\bar{\lambda})=\mathcal{H}^2(\mathbb{D}).$$

(3) Suppose that
$$span\{f_z;z\in\mathbb{D}\}=span\{\frac{1}{1-\bar{z}s};z\in\mathbb{D}\}$$
is not dense in $\mathcal{H}^2(\mathbb{D})$,
by $0\neq\phi\in\mathcal{H}^{\infty}(\mathbb{D})$ and by the definition of reproducing kernel $f_z$,
we get that there exists $0\neq g\in\mathcal{H}^2(\mathbb{D})$,
for any given $z\in\mathbb{D}$,
we have
\begin{eqnarray*}
0=<g,\bar{\phi}(z)f_z>_{\mathcal{H}^2}=\phi(z)g(z)=<\phi g,f_z>_{\mathcal{H}^2}.
\end{eqnarray*}

So we get $g=0$ by the Analytic Continuation Theorem \cite{HenriCartanyujiarong2008}P28,
that is a contradiction for $g\neq0$.
Therefor we get that $span\{f_z;z\in\mathbb{D}\}$
is dense in $\mathcal{H}^2(\mathbb{D})$,
that is,
$$\bigvee\ker\limits_{\bar{\lambda}\in\phi(\mathbb{D})}(M_{\phi}^{*}-\bar{\lambda})=\mathcal{H}^2(\mathbb{D}).$$

(4) By Definition $\ref{nyejiexihanshudedingyijitoubujiaobu51}$ and the conditions of this theorem,
for any given $\lambda\in\phi(\mathbb{D})$,
there exists $z_0\in\mathbb{D}$, exists $m$-th polynomial $p_m(z)$ and outer function $h(z)$ such that
\begin{eqnarray*}
\left.
 \begin{array}{l}
\phi(z)-\lambda=\phi(z)-\phi(z_0)=p_m(z)h(z),
 \end{array}
 \right.
\end{eqnarray*}

We give $\dim\ker(M_{\phi(z)}^{*}-\bar{\lambda})=m$ by the following $(i)(ii)(iii)$ assertions.

(i) Let the roots of $p_m(z)$ be $z_0,z_1,\cdots,z_{m-1}$,
then there exists decomposition
$$p_m(z)=(z-z_0)(z-z_1)\cdots(z-z_{m-1}),$$
and denote that $p_{m,z_0z_1\cdots z_{m-1}}(z)$ is the decomposition of $p_m(z)$ by the permutation of $z_0,z_1,\cdots,z_{m-1}$,
then to get
$$\dim\ker M_{p_{m,z_0z_1\cdots z_{m-1}}}^{*}=m.$$

By the Taylor expansions of functions in $\mathcal{H}^2(\mathbb{T})$,
we get a naturally isomorphic
\begin{eqnarray*}
\left.
 \begin{array}{l}
F_{s}:\mathcal{H}^2(\mathbb{D})\to\mathcal{H}^2(\mathbb{D}-s),F_{s}(g(z))\to g(z+s),\text{ÆäÖÐ}s\in\mathbb{C}.
 \end{array}
 \right.
\end{eqnarray*}

It is easy to get that $G=\{F_s;s\in\mathbb{C}\}$ is a abelian group by the composite operation $\circ$,
hence for $0\leq n\leq m-1$, there is
\begin{displaymath}
\begin{array}{rcl}
\mathcal{H}^2(\mathbb(D)) & \underrightarrow{\qquad\quad~M_{z-z_n}~~\quad\qquad} & \mathcal{H}^2(\mathbb(D))\\
F_{z_n}\downarrow &      & \downarrow F_{z_n}\\
\mathcal{H}^2(\mathbb{D}-z_n)&\overrightarrow{\qquad\qquad~M_{z}^{'}\qquad\qquad}& \mathcal{H}^2(\mathbb{D}-z_n)
\end{array}
\end{displaymath}

Let $T$ be the backward shift operator on the Hilbert space $\mathcal{L}^2(\mathbb{N})$,
that is,
$$T(x_1,x_2,\cdots)=(x_2,x_3,\cdots).$$
With the naturally isomorphic between $\mathcal{H}^2(\mathbb{D}-z_n)$ and $\mathcal{H}^2(\partial(\mathbb{D}-z_n))$,
$M_{z}^{'*}$ is equivalent to the backward shift operator $T$ on $\mathcal{H}^2(\partial(\mathbb{D}-z_n))$,
that is,
$M_{z}^{'*}$ is a surjection and
$$\dim\ker M_{z}^{'*}=1,$$
hence $M_{z-z_n}^{*}$ is a surjection and
$$\dim\ker M_{z-z_n}^{*}=1,$$
where $0\leq n\leq m-1$.

By the composition of
$$F_{z_{m-1}}\circ F_{z_{m-2}}\circ\cdots\circ F_{z_{0}},$$
$M_{p_{m,z_0z_1\cdots z_{m-1}}}^{'*}$ is equivalent to $T^m$.
that is, $M_{p_{m,z_0z_1\cdots z_{m-1}}}^{'*}$ is a surjection and
$$\dim\ker M_{p_{m,z_0z_1\cdots z_{m-1}}}^{'*}=m,$$
hence $M_{p_{m,z_0z_1\cdots z_{m-1}}}^{*}$ is a surjection and
$$\dim\ker M_{p_{m,z_0z_1\cdots z_{m-1}}}^{*}=m.$$

(ii) Because $\mathcal{H}^{\infty}$ is a abelian Banach algebra,
$M_{p_{m}}$ is independent to the permutation of $1$-th factors of $p_m(z)$,
that is,
$M_{p_{m}}^{*}$ is independent to the $1$-th factors multiplication of $p_m(z)=(z-z_0)(z-z_1)\cdots(z-z_{m-1})$.

Because $G=\{F_s;s\in\mathbb{C}\}$ is a abelian group by composition operation $\circ$,
for $0\leq n\leq m-1$,
$F_{z_{m-1}}\circ F_{z_{m-2}}\circ\cdots\circ F_{z_{0}}$ is independent to the permutation of composition.
Hence $M_{p_{m}}^{*}$ is a surjection and
\begin{eqnarray*}
\left.
 \begin{array}{l}
\dim\ker M_{p_{m}}^{*}=\dim\ker M_{p_{m,z_0z_1\cdots z_{m-1}}}^{*}=m.
 \end{array}
 \right.
\end{eqnarray*}

(iii) By Definition $\ref{waihanshudingyi46}$ and
Theorem $\ref{chengfasuanzishimanshedangqiejindangshiwaihanshu45}$,
also by \cite{JohnBConway1990}P93,
\cite{zhanggongqinglinyuanqu2006}P97 and \cite{JohnBConway1990}P305 we get that the multiplication operator $M_{h}$ is surjection that associated with the outer function $h$.
Hence we get
\begin{eqnarray*}
\left.
 \begin{array}{l}
\ker M_{h(z)}^{*}=(ranM_{h(z)})^{\bot}=(\mathcal{H}^2(\mathbb{D}))^{\bot}=0.
 \end{array}
 \right.
\end{eqnarray*}

Because there exists decomposition $M_{p_m(z)h(z)}^{*}=M_{h(z)}^{*}M_{p_m(z)}^{*}$ on $\mathcal{H}^2(\mathbb{D})$, we get
\begin{eqnarray*}
\left.
 \begin{array}{l}
\dim\ker(M_{\phi}^{*}-\bar{\lambda})=\dim\ker M_{p_m(z)h(z)}^{*}=\dim\ker M_{p_m(z)}^{*}=m.
 \end{array}
 \right.
\end{eqnarray*}

By (1)(2)(3)(4) we get the adjoint multiplier operator $M_{\phi}^{*}\in\mathcal{B}_m(\bar{\phi}(\mathbb{D}))$.
\end{proof}

By Theorem $\ref{cowendouglussuanziyujiexichengfasuanzijigezhonghundundeguanxi24}$
and Lemma $\ref{waihanshukeyidexingzhi49}$,
we get
\begin{corollary}\label{nyejiexihanshushicowendouglashanshudetiaojianqigenhanshukeni52}
Let $\phi\in\mathcal{H}^{\infty}(\mathbb{D})$ be a $m$-folder complex analytic function,
for any given $z_0\in\mathbb{D}$,
if the rooter function of $\phi$ at $z_0$ is invertible in the Banach algebra $\mathcal{H}^{\infty}(\mathbb{D})$,
then $\phi$ is a Cowen-Douglas function.
Especially,
for any given $n\geq1$,
non-constant $p_n(z)=\sum\limits_{k=0}^{n}a_kz^k\in\mathcal{H}^{\infty}(\mathbb{D})$ is a Cowen-Douglas function.
\end{corollary}

\section{Application of Cowen-Douglas function  and chaos}

This part we give some properties about the adjoint multiplier of Cowen-Douglas function,
such as application on the inverse chaos problem for $\phi(T)$,
where $\phi$ is a Cowen-Douglas function and $T$ is the backward shift operator on $\mathcal{L}^2(\mathbb{N})$,
i.e.,$T(x_1,x_2,\cdots)$
$=(x_2,x_3,\cdots)$.

\begin{theorem}\label{hardykongjianshangchengfasuanziyuyuanzhoujiaofeikongdedingli17}
If $\phi\in\mathcal{H}^{\infty}(\mathbb{D})$ is a Cowen-Douglas function,
$M_{\phi}$ is the multiplication by $\phi$ on $\mathcal{H}^{2}(\mathbb{D})$,
Then the following assertions are equivalent

(1) $M_{\phi}^{*}$ is Devaney chaotic;

(2) $M_{\phi}^{*}$ is distributionally chaotic;

(3) $M_{\phi}^{*}$ is strong mixing;

(4) $M_{\phi}^{*}$ is Li-Yorke chaotic;

(5) $M_{\phi}^{*}$ is hypercyclic;

(6) $\phi(\mathbb{D})\bigcap\mathbb{T}\neq\emptyset$.
\end{theorem}
\begin{proof}

(i) By \cite{FBayartEMatheron2009}P138 we get that $M_{\phi}^{*}$ is Devaney chaotic if and only if it is hypercyclic,
i.e., $\phi$ is non-constant and $\phi(\mathbb{D})\bigcap\mathbb{T}\neq\emptyset$,
hence
(1)$\Longleftrightarrow$(5)$\Longleftrightarrow$(6).

(ii) To get that (6) imply (1)(2)(3)(4). Because $\phi\in\mathcal{H}^{2}(\mathbb{D})$ is a Cowen-Douglas function,
by Definition $\ref{youjiejiexihanshuliCowenDouglashanshudingyi47}$,
$M_{\phi}^{*}\in\mathcal{B}_n(\bar{\phi}(\mathbb{D}))$.

By Theorem $\ref{cowendouglassuanzidingyidevaneyliyrokefenbuqianghunhedingli2}$
we get that if $\phi(\mathbb{D})\bigcap\mathbb{T}\neq\emptyset$,
then (1)(2)(3) is valid.
On Banach spaces Devaney chaotic, distributionally chaotic and strong mixing imply Li-Yorke chaotic,
respectively.
Hence (4) is valid.
Because $\bar{\phi}(\mathbb{D})\bigcap\mathbb{T}\neq\emptyset$ and
$\phi(\mathbb{D})\bigcap\mathbb{T}\neq\emptyset$ are mutually equivalent, (6) imply (1)(2)(3)(4).

(iii) To get that (1)(2)(3)(4) imply (6).
By (1)(2)(3) imply (4), respectively,
it is enough to get that (4) imply (6).

If $M_{\phi}^{*}$ is Li-Yorke chaotic,
then $\phi$ is non-constant and by \cite{HouBLiaoGCaoY2012} theorem3.5 we get
$\sup\limits_{n\to+\infty}\|M_{\phi}^{*n}\|\to\infty,$
hence
$$
\|M_{\phi}\|=\|M_{\phi}^{*}\|>1,
\quad \text{that is},\quad
\sup\limits_{z\in\mathbb{D}}|\phi(z)|>1.
$$
Moreover,
we get
$$
\inf\limits_{z\in\mathbb{D}}|\phi(z)|<1,
$$
Indeed,
if we assume that
$$\inf\limits_{z\in\mathbb{D}}|\phi(z)|\geq1,$$
then
$$
\frac{1}{\phi}\in\mathcal{H}^{\infty},\quad\text{and}\quad\|M_{\frac{1}{\phi}}^{*}\|=\|M_{\frac{1}{\phi}}\|\leq1.
$$
Hence, for any $0\neq f\in\mathcal{H}^2(\mathbb{D})$, we get

$$
\|M_{\phi}^{*n}f\|\geq\frac{1}{\|M_{\phi}^{*-n}\|}\|f\|\geq\frac{1}{\|M_{\frac{1}{\phi}}^{*}\|^n}\|f\|\geq\|f\|.
$$
It is a contradiction to $M_{\phi}^{*}$ is Li-Yorke chaotic.

Therefor $M_{\phi}^{*}$ is Li-Yorke chaotic imply
$$
\inf\limits_{z\in\mathbb{D}}|\phi(z)|<1<\sup\limits_{z\in\mathbb{D}}|\phi(z)|.
$$
By the properties of a simple connectedness argument of complex analytic functions we get $\phi(\mathbb{D})\bigcap\mathbb{T}\neq\emptyset$.
Hence we get (4) imply (6), that is (1)(2)(3)(4) both imply (6).
\end{proof}

\begin{corollary}\label{cowendouglussuanzideniyuhardykongjiandechengfasuanzi23}
If $\phi\in\mathcal{H}^{\infty}(\mathbb{D})$ is a invertible Cowen-Douglas function in the Banach algebra $\mathcal{H}^{\infty}(\mathbb{D})$,
and let $M_{\phi}$ be the multiplication by $\phi$ on $\mathcal{H}^{2}(\mathbb{D})$.
Then $M_{\phi}^{*}$ has property $\alpha$ if and only if $M_{\phi}^{*-1}$ has,
where $\alpha\in$
\{Devaney chaotic, distributionally chaotic, strong mixing, Li-Yorke chaotic, hypercyclic\}.
\end{corollary}
\begin{proof}
Because of $T=(T^{-1})^{-1}$,
it is enough to prove that the dynamic properties of $M_{\phi}^{*}$ is imply
the same dynamic properties of $M_{\phi}^{*-1}$.

By Definition $\ref{youjiejiexihanshuliCowenDouglashanshudingyi47}$ we get
$M_{\phi}^{*}\in\mathcal{B}_n(\bar{\phi}(\mathbb{D}))$,
with a simple computing we get $M_{\phi}^{*-1}\in\mathcal{B}_n(\cfrac{1}{\bar{\phi}}(\mathbb{D}))$.

If $M_{\phi}^{*}$ has property $\alpha$,
by Theoem $\ref{hardykongjianshangchengfasuanziyuyuanzhoujiaofeikongdedingli17}$ we get $\phi(\mathbb{D})\bigcap\mathbb{T}\neq\emptyset$,
and by the properties of complex analytic functions we get $\cfrac{1}{\phi}(\mathbb{D})\bigcap\mathbb{T}\neq\emptyset$.
Because of $M_{\phi}^{*-1}\in\mathcal{B}_n(\cfrac{1}{\bar{\phi}}(\mathbb{D}))$ and by Theorem $\ref{hardykongjianshangchengfasuanziyuyuanzhoujiaofeikongdedingli17}$ we get
$M_{\phi}^{*-1}$ has property $\alpha$.
\end{proof}

We now study some properties about scalars perturbation of an operator inspired by \cite{HoubingzheTiangengShiluoyi2009} and \cite{BermudezBonillaMartinezGimenezPeiris2011} that research some properties about the compact perturbation of scalar operator.

\begin{definition}[\cite{HouBingzheLuoLvlin2017}]\label{shuzhijiafasuanzidedingyi1}
Let $T\in\mathcal{B}(\mathbb{H})$ and define

(i) $S_{LY}(T)=\{\lambda\in\mathbb{C};\lambda +T \text{ is Li-Yorke chaotic} \};$

(ii) $S_{DC}(T)=\{\lambda\in\mathbb{C};\lambda +T \text{ is distributionally chaotic}\};$

(iii) $S_{DV}(T)=\{\lambda\in\mathbb{C};\lambda +T \text{ is Devaney chaotic}\};$

(iv) $S_{H}(T)=\{\lambda\in\mathbb{C};\lambda +T \text{ is hypercyclic}\}$.
\end{definition}

By Definition $\ref{shuzhijiafasuanzidedingyi1}$ we get $S_{LY}(\lambda +T)=S_{LY}(T)-\lambda$,
$S_{DC}(\lambda +T)=S_{DC}(T)-\lambda$,
$S_{DV}(\lambda +T)=S_{DV}(T)-\lambda$ and $S_{H}(\lambda +T)=S_{H}(T)-\lambda$.

\begin{example}\label{shuzhipukeyishikaiji5}
Let $T$ be the backward shift operator on $\mathcal{L}^2(\mathbb{N})$,
$T(x_1,x_2,\cdots)$
$=(x_2,x_3,\cdots)$.
Then
\begin{eqnarray*}
&&S_{LY}(T)=S_{DC}(T)=S_{DV}(T)=S_{H}(T)=2\mathbb{D}\setminus\{0\},\\
&&S_{LY}(2T)=S_{DC}(2T)=S_{DV}(2T)=S_{H}(2T)=3\mathbb{D}.
\end{eqnarray*}
Hence $S_{LY}(T)$ and $S_{LY}(2T)$ are open sets.
\end{example}
\begin{proof}
By \cite{JohnBConway1990}P209 we get $\sigma(T)=cl\mathbb{D}$ and $\sigma(2T)=cl2\mathbb{D}$,
by Definition $\sigma(T)$ we get $\sigma(\lambda +T)=\lambda+cl\mathbb{D}$.
Because of the method to prove the conclusion is similarly for $T$ and $2T$,
we only to prove the conclusion for $T$.

By the naturally isomorphic between $\mathcal{H}^{2}(\mathbb{T})$ and $\mathcal{H}^{2}(\mathbb{D})$.
Let $\mathcal{L}^2(\mathbb{N})=\mathcal{H}^{2}(\mathbb{T})$,
by the definition of $T$ we get $(\lambda +T)^{*}$ is the multiplication operator $M_{f}$ by $f(z)=\bar{\lambda}+z$
on the Hardy space $\mathcal{H}^{2}(\mathbb{T})$.
By the Dirichlet Problem \cite{HenriCartanyujiarong2008}P103 we get that $f(z)$ is associated with the complex analytic function $\phi(z)=\bar{\lambda}+z\in\mathcal{H}^{\infty}(\mathbb{D})$ determined by the boundary condition $\phi(z)|_{\mathbb{T}}=f(z)$.

By Corollary $\ref{nyejiexihanshushicowendouglashanshudetiaojianqigenhanshukeni52}$ we get that $\phi$ is a Cowen-Douglas function.
Therefor by the natural isomorphic between $\mathcal{H}^{2}(\mathbb{T})$ and $\mathcal{H}^{2}(\mathbb{D})$,
$\lambda+T$ is naturally equivalent to the operator $M_{\phi}^{*}$ on $\mathcal{H}^{2}(\mathbb{D})$.

By Theorem $\ref{hardykongjianshangchengfasuanziyuyuanzhoujiaofeikongdedingli17}$
we get that $M_{\phi}^{*}$ is hypercyclic or Devaney chaotic or distributionally chaotic or Li-Yorke chaotic if and only if $\phi(\mathbb{D})\bigcap\mathbb{T}\neq\emptyset$.

Because of $\sigma(\lambda +T)=\sigma(\bar{\lambda}+T^{*})$,
we get
$$\sigma(\lambda +T)=\sigma(M_{\phi}^{*})=\sigma(M_{\phi})\supseteq\phi(\mathbb{D}),$$
hence
$$S_{LY}(T)=S_{DC}(T)=S_{DV}(T)=S_{H}(T)=2\mathbb{D}\setminus\{0\}$$ is an open set.
\end{proof}

\begin{corollary}\label{cowendouglusjiaquanyiweisuansuanzideniyuhardykongjiandechengfasuanzi223}
If $\phi\in\mathcal{H}^{\infty}(\mathbb{D})$ is a invertible Cowen-Douglas function in the Banach algebra $\mathcal{H}^{\infty}(\mathbb{D})$,
and let $T$ be the backward shift operator on $\mathcal{L}^2(\mathbb{N})$,
Then $\phi(T)$ has property $\alpha$ if and only if $\phi^{-1}(T)$ has,
where $\alpha\in$
\{Devaney chaotic, distributionally chaotic, strong mixing, Li-Yorke chaotic, hypercyclic\}.
\end{corollary}

\begin{corollary}
Let $T$ be the backward shift operator on $\mathcal{L}^2(\mathbb{N})$,
For any given $n\in\mathbb{N}$,
if non-constant $p_n(T)=\sum\limits_{k=0}^{n}a_kT^k$ is invertible,
then $p_n(T)$ has property $\alpha$ if and only if $(p_n(T))^{-1}$ has,
where $\alpha\in$
\{Devaney chaotic, distributionally chaotic, strong mixing, Li-Yorke chaotic, hypercyclic\}.
\end{corollary}

\section{Acknowledgments}
The author would like to thank the referee for his/her careful reading of the paper and helpful comments and suggestions. Also, the author thanks Professor Bingzhe Hou for helpful discussions on theorem~\ref{hardykongjianshangchengfasuanziyuyuanzhoujiaofeikongdedingli17}
and thanks Puyu Cui for helpful techniques (1) of theorem~\ref{cowendouglussuanziyujiexichengfasuanzijigezhonghundundeguanxi24}.

\bibliographystyle{amsplain}

\begin{thebibliography}{99}
\bibitem{BermudezBonillaMartinezGimenezPeiris2011}
Berm¨²dez, Bonilla, Mart¨ªnez-Gim¨¦nez and A. Peiris, \textit{Li-Yorke and distributionally chaotic operators}, J. Math. Anal. Appl. \textbf{373}(2011), 1-83-93.

\bibitem{BHouPCuiYCao2010}B. Hou, P. Cui and Y. Cao, \textit{Chaos for Cowen-Douglas operators}, Pro. Amer. Math. Soc. \textbf{138}(2010), 926-936.

\bibitem{DStockman2012}D. Stockman, \textit{Li-Yorke Chaos in Models with Backward Dynamics}, (2012), available at http://sites.udel.edu/econseminar/files/2012/03/lyc-backward.pdf

\bibitem{FBayartEMatheron2009}F. Bayart and E. Matheron, \textit{Dynamics of Linear Operators}, Cambridge University Press, 2009.


\bibitem{HenriCartanyujiarong2008}
Henri Cartan(Translated by Yu Jiarong), \textit{Theorie elementaire des fonctions analytuques d'une ou plusieurs variables complexes}, Higer Education Press, Peking, 2008.

\bibitem{HouBLiaoGCaoY2012}
Hou B, Liao G and Cao Y, \textit{Dynamics of shift operators}, Houston Journal of Mathematics \textbf{38}(2012), 4-1225-1239.

\bibitem{HouBingzheLuoLvlin2016}Hou Bingzhe and Luo Lvlin, \textit{Li-Yorke chaos for invertible mappings on noncompact spaces}, Turk. J. Math. \textbf{40} (2016), 411-416.

\bibitem{HouBingzheLuoLvlin2017}Hou Bingzhe and Luo Lvlin, \textit{Li-Yorke chaos translation set for linear operators}, arXiv:1712.02580.

\bibitem{HoubingzheTiangengShiluoyi2009}
Hou Bingzhe,Tian Geng and Shi Luoyi, \textit{Some Dynamical Properties For Linear Operators}, Illinois Journal of Mathematics, Fall 2009.

\bibitem{JohnBConway1990}John B. Conway, \textit{A Course in Functional Analysis}, Second Edition, Springer-Verlag New York, 1990.

\bibitem{JohnBGarnett2007}
John B. Garnett, \textit{Bounded Analytic Functions}, Revised First Edition, Springer-Verlag, 2007.

\bibitem{KennethHoffman1962}
Kenneth Hoffman, \textit{Banach Spaces of Analytic Functions}, Prentice-Hall, Inc. , Englewood Cliffs, N. J. , 1962.

\bibitem{LuoLvlinHouBingzhe2016}Luo Lvlin and Hou Bingzhe, \textit{Li-Yorke chaos for invertible mappings on compact metric spaces}, Arch. Math. 2016, DOI 10.1007/s00013-016-0972-5.

\bibitem{LuoLvlinHouBingzhe2015}Luo Lvlin and Hou Bingzhe, \textit{Some remarks on distributional chaos for bounded linear operators}, Turk. J. Math. \textbf{39} (2015), 251-258.

\bibitem{RonaldGDouglas1998}
Ronald G. Douglas, \textit{Banach Algebra Techniques in Operator Theory}, Second Edition, Springer-Verlag New York, 1998.


\bibitem{WilliamArveson2002}William Arveson, \textit{A Short Course on Spectral Theory}, Springer Science+Businee Media, LLC, 2002.

\bibitem{zhanggongqinglinyuanqu2006}
Zhang Gongqing and Lin Yuanqu, \textit{Lecture notes of functional analysis(Volume 1)}, Peking University Press, Peking, 2006.
\end{thebibliography}

\end{document}